\documentclass[10pt,reqno]{amsart}
\usepackage{amsmath}
\usepackage{amssymb}
\usepackage{amsthm}
\usepackage{epsfig}
\usepackage{graphicx,color}
\usepackage{xcolor}
\usepackage[pdftex]{hyperref}

\newlength{\defbaselineskip}
\newcommand{\setlinespacing}[1]%
{\setlength{\baselineskip}{#1 \defbaselineskip}}

\numberwithin{equation}{section}

\newtheorem{thm}{Theorem}[section]
\newtheorem{cor}[thm]{Corollary}
\newtheorem{lem}[thm]{Lemma}
\newtheorem{prop}[thm]{Proposition}

\theoremstyle{definition}

\theoremstyle{remark}
\newtheorem{rem}[thm]{Remark}
\numberwithin{equation}{section}

\makeatletter
\@namedef{subjclassname@2020}{\textup{2020} Mathematics Subject Classification}
\makeatother

\begin{document}

\title[Strichartz and uniform Sobolev inequalities]
{Strichartz and uniform Sobolev inequalities \\ for the elastic wave equation}

\author{Seongyeon Kim, Yehyun Kwon, Sanghyuk Lee and Ihyeok Seo}

\thanks{This work was supported by a KIAS Individual Grant (MG082901) at Korea Institute for Advanced Study (Kim), a KIAS Individual Grant (MG073702) at Korea Institute for Advanced Study and NRF-2020R1F1A1A01073520 (Kwon), NRF-2021R1A2B5B02001786 (Lee) and NRF-2022R1A2C1011312 (Seo).}


\address{School of Mathematics, Korea Institute for Advanced Study, Seoul 02455, Republic of Korea}
\email{synkim@kias.re.kr}
\email{yhkwon@kias.re.kr}

\address{Department of Mathematical Sciences and RIM, Seoul National University, Seoul 08826, Republic of Korea}
\email{shklee@snu.ac.kr}

\address{Department of Mathematics, Sungkyunkwan University, Suwon 16419, Republic of Korea}
\email{ihseo@skku.edu}

\begin{abstract}
We prove dispersive estimate for the elastic wave equation by which we extend the known Strichartz estimates for the classical wave equation to those for the elastic wave equation. In particular, the endpoint Strichartz estimates are deduced.  For the purpose we diagonalize the symbols of the Lam\'e operator and its semigroup, which also gives an alternative and simpler proofs of the previous results on perturbed elastic wave equations. Furthermore, we obtain uniform Sobolev inequalities for the elastic wave operator.
\end{abstract} 

\maketitle

\section{Introduction}
Let $n\ge2$ and let $f, g\colon \mathbb{R}^n \rightarrow \mathbb{C}^n$ and $F\colon \mathbb R\times \mathbb R^{n} \rightarrow \mathbb C^n$ be vector fields. We consider the inhomogeneous elastic wave equation
\begin{equation}\label{e:elastic}
\begin{cases}
	(\partial_t^2 -\Delta^\ast) u(t,x) = F(t,x) , \\ 
	u(0,x)=f(x), \quad \partial_t u(0,x)=g(x),
\end{cases}
\end{equation}
where $\Delta^\ast$ denotes the Lam\'e operator (acting in the spatial variable $x$) defined by
\begin{equation*} 
	\Delta^\ast u =\mu\Delta u +(\lambda + \mu)\nabla \mathrm{div} u,
\end{equation*}
and the Laplacian $\Delta$ acts on a vector field component-wise. The Lam\'e constants $\lambda,\mu\in\mathbb{R}$ satisfy the standard condition 
\begin{equation} \label{e:elliptic}
	\mu >0, \quad \lambda + 2\mu >0,
\end{equation}
which guarantees the ellipticity of $\Delta^\ast$. The equation has been used to model wave propagation in an elastic medium, where $u$ denotes the displacement field of the medium (see e.g., \cite{LL,MR1262126}).

\subsubsection*{Dispersive estimate and Strichartz estimate}
We begin by introducing the notations. For a vector-valued function $f = (f_1, \ldots ,f_n)\colon \mathbb R^n\to \mathbb C^n$ we use the norms
\begin{equation}\label{e:def1}
	\|f\|_{L^r(\mathbb R^n)} =
	\begin{cases}
		\big(\sum_{j=1}^n \|f_j\|_{L^r(\mathbb R^n)}^r\big)^{\frac1r}, & 1\le r<\infty, \\
		\max_{1\le j\le n} \|f_j\|_{L^\infty(\mathbb R^n)}, & r=\infty.
	\end{cases}
\end{equation}
For a time dependent vector field $u=(u_1, \ldots, u_n)\colon \mathbb R\times \mathbb R^n \to \mathbb C^n$ we denote 
\begin{equation*}
	\|u\|_{L_t^qL_x^r(\mathbb R\times \mathbb R^n)} = \big\| \| u(t,\cdot) \|_{L^r(\mathbb R^n)} \big\|_{L^q_t(\mathbb R)}.
\end{equation*} 
Besides, we use the standard notations $\|f\|_{\dot H^\sigma_r}=\||\nabla|^\sigma f\|_{L^r}$ and $\|f\|_{\dot H^\sigma}=\|f\|_{\dot H^\sigma_2}$. For $q,r\ge2$, $r\neq \infty$ and $(q,r,n)\neq(2,\infty, 3)$, we say $(q,r)$ is \emph{wave-admissible} if $\frac1q\le \frac{n-1}2(\frac12-\frac1r)$. We also call $(q,r)$  \emph{sharp wave-admissible} if $\frac1q= \frac{n-1}2(\frac12-\frac1r)$.

Our first result is the following dispersive estimate.
\begin{thm} \label{t:dispersive}
Suppose $\widehat f$ is supported in the annulus $\{\xi \in \mathbb R^n\colon 1/2 \le |\xi|\le2\}$. Then
\begin{equation}\label{e:dispers}
	\|e^{it\sqrt{-\Delta^\ast}}f\|_{L^\infty(\mathbb R^n)} \lesssim |t|^{-\frac{n-1}2}\|f\|_{L^1(\mathbb R^n)}.
\end{equation}
\end{thm}

Once we have the dispersive estimate \eqref{e:dispers} we can prove the Strichartz estimate for the elastic wave equation whenever $(q,r)$ is wave-admissible using the abstract framework due to Keel--Tao \cite{MR1646048}.  Indeed, we combine the frequency-localized dispersive estimate \eqref{e:dispers} and the $L^2$ estimate (see \eqref{e:unitary} in Section \ref{sec:pfofmain}) to get the estimate for frequency-localized initial data, then by scaling and the Littlewood--Paley inequality we obtain the homogeneous Strichartz estimates with an arbitrary initial data $(f,g)\in \dot H^s(\mathbb R^n)\times \dot H^{s-1}(\mathbb R^n)$. The estimates and the standard $TT^\ast$-argument give the following. 

\begin{thm} \label{t:strichartz}
Let $(q,r)$ and $(\tilde q, \tilde r)$ be wave-admissible pairs with $r, \tilde r<\infty$. If $u$ is a solution to the Cauchy problem \eqref{e:elastic}, then we have
\begin{equation}\label{e:strichartz}
	\|u\|_{L_t^qL_x^r(\mathbb R\times \mathbb R^n)} \lesssim \|f\|_{\dot H^s(\mathbb R^n)}+\|g\|_{\dot H^{s-1}(\mathbb R^n)} +\|F\|_{L_t^{\tilde q'}L_x^{\tilde r'}(\mathbb R\times \mathbb R^n)}
\end{equation}
provided that 
\begin{equation*}
	\frac 1q+\frac nr=\frac n2-s=\frac 1{\tilde q'}+\frac n{\tilde r'}-2.
\end{equation*}
\end{thm}

\begin{rem}\label{r:inhomo} 
When $f=g=0$ in \eqref{e:elastic}, the estimate \eqref{e:strichartz}  is called the inhomogeneous Strichartz estimate, which holds for a wider range of pairs $(q,r)$ and $(\tilde q, \tilde r)$ than that of the wave-admissible pairs. As the precise description of the pairs is rather complicated, we provide the detailed statement in Section \ref{sec:inhomo} (see Theorem \ref{exinho}).
\end{rem}

The estimates \eqref{e:dispers} and \eqref{e:strichartz} are in complete analogue with  the dispersive estimate and Strichartz estimate, respectively, for the classical wave equation
\begin{equation}\label{e:wave}
\begin{cases}
	(\partial_t^2 -\Delta) u(t,x) = F(t,x), \\
	u(0)=f, \quad \partial_t u(0)=g.
\end{cases}
\end{equation}
In context of the wave equation \eqref{e:wave} all the above results are standard, and there has been a large body of literature concerning the Strichartz estimates for \eqref{e:wave}.  Among others, the diagonal case $q = r$ was obtained in \cite{MR512086} in connection with the restriction theorems for the cone.  This was later extended to mixed norms $L^q_t L^r_x$ independently by Ginibre and Velo \cite{MR1351643} and Lindblad and Sogge \cite{MR1335386}.  The remaining endpoints where $q=2$ were later settled by Keel and Tao \cite{MR1646048}.

\subsubsection*{Perturbed elastic wave equation} 
Now we turn to the perturbed elastic wave equation
\begin{equation}\label{e:perturb}
\begin{cases}
	(\partial_t^2 -\Delta^\ast + V(x)) u(t,x) = 0, \\
	u(0,x)=f(x), \quad \partial_t u(0,x)=g(x).
\end{cases}
\end{equation}
In \cite{MR3361693} Barcel\'o et al. studied \eqref{e:perturb} in three spatial dimension $n=3$. Under the assumptions that $V\colon \mathbb R^3\to \mathcal M_{3\times 3}(\mathbb R)$ is symmetric, $|x|^2|V(x)|\le c$ for a small constant $c$, $\mu>0$ and $3\lambda+2\mu>0$, they obtained the Strichartz estimates for \eqref{e:perturb}
\begin{equation*}
	\|u\|_{L_t^q \dot H_r^{\frac1r-\frac1q}} \lesssim \|f\|_{\dot H^\frac12} +\|g\|_{\dot H^{-\frac12}}
\end{equation*}
whenever $(q,r)$ is sharp wave-admissible. 

In \cite{KKS} the authors improved the result to a more general class of potentials $V$ under the weaker assumption \eqref{e:elliptic} on the Lam\'e coefficients.  To facilitate the statement, let us recall the Fefferman--Phong class defined by
\begin{equation*}
	\mathcal F^p:=\bigg\{ V\colon\mathbb R^n\to \mathcal M_{n\times n}(\mathbb C)\colon \|V\|_{\mathcal{F}^p} =\sup_{x\in \mathbb{R}^n,r>0} r^{2-\frac np} \bigg( \int_{B(x,r)}|V(y)|^p dy\bigg)^{\frac1p} < \infty\bigg\},
\end{equation*}
for $1 \leq p \leq n/2$.  Here $B(x,r)$ denotes the open ball in $\mathbb R^n$ centered at $x$ with radius $r$ and $|V|=\big(\sum_{i,j=1}^n |V_{ij}|^2\big)^{1/2}$.  If $1\le p<n/2$ the class $\mathcal{F}^p$ includes the weak space $L^{\frac n2,\infty}$, and in particular, the critical inverse-square potential $|x|^{-2}$. For related results on the classical wave equation, see \cite{MR2003358,MR2106340,MR4061548} and references therein.

\begin{thm}[\cite{KKS}] \label{t:str-per}
Let $n\ge3$ and $V \in \mathcal{F}^p$ for $p>\frac{n-1}2$.  Let $u$ be a solution to \eqref{e:perturb} with $(f,g)\in \dot H^{1/2} (\mathbb{R}^n) \times \dot H^{-1/2}(\mathbb{R}^n)$.  If $\|V\|_{\mathcal F^p}\le c$ for $c>0$ small enough,  then
\begin{equation*} 
	\|u\|_{L_t^q \dot H^{\sigma}_{r}} \lesssim \|f\|_{\dot H^{1/2}} + \|g\|_{\dot H^{-1/2}}
	\end{equation*}
whenever $(q,r)$ is wave-admissible, $q>2$ and $\sigma=\frac{1}{q} + \frac{n}{r} - \frac{n-1}{2}$.
\end{thm}

In \cite{KKS}, instead of using the spectral theoretic approach of \cite{MR3361693}, the authors took an approach based on harmonic analysis technique. They focused on the Fourier multiplier of $\sqrt{-\Delta^\ast}$ after viewing the solution  as a sum of the Duhamel term given by $V(x)u(t,x)$  and the solution to the free case.

The approach in \cite{KKS} is significantly different from that in \cite{MR3361693} and leads to improvements on assumptions on $\lambda$, $\mu$ and $V$. However, it still has a similar flavor in that it followed a strategy making use of the Helmholtz decomposition (i.e., Leray projection) of vector fields $f=f_S+f_P$, where $f_S$ is a divergence-free field and $f_P$ is a gradient field.  To carry out the strategy, it was necessary to control $\|f_S\|_{\dot H^s}+\|f_P\|_{\dot H^s}$ and $\|F_S\|_{L^2(w)}+\|F_P\|_{L^2(w)}$ with $\|f\|_{\dot H^s}$ and $\|F\|_{L^2(w)}$, respectively. Thus, the authors had  to use $L^2$-orthogonality between the Leray projections (\cite[Lemma 2.1]{KKS}) and elliptic regularity estimates (\cite[Lemma 4.1]{KKS}).

In this paper, we develop a new approach to analyze \eqref{e:elastic} and \eqref{e:perturb}.  Rather than relying on the Helmholtz decomposition of vector fields,  we make use of  diagonalization of the Lam\'e operator. As the rotations resulting from the diagonalization process are smooth and homogeneous of degree zero, they satisfy the Mikhlin condition, hence, the classical multiplier theorems become available. Therefore we can utilize the estimates for the classical wave equations.  The diagonalization argument completely replaces the role of the Helmholtz decomposition,  for instance, we can prove Theorem \ref{t:str-per} in a simpler way without using the Helmholtz decomposition (see Section \ref{sec:4}). Similarly, the proofs of all the other results in \cite{MR3361693, KKS} can be simplified. This direct and Fourier-analytic approach is likely to be useful in different problems because neither any orthogonality property of the Leray projections nor elliptic regularity estimate is necessary.

\subsubsection*{Uniform Sobolev inequality}
We can adapt the diagonalization method to obtain the uniform Sobolev inequality for the elastic wave operator $\partial_t^2-\Delta^\ast$. 

For the wave operator $\partial_t^2-\Delta$, 
Kenig, Ruiz and Sogge \cite{MR894584} proved the uniform Sobolev inequality  
\begin{equation}\label{e:uni-sob}
	\|u\|_{L^q(\mathbb R^{1+n})} \le C\| (\partial_t^2-\Delta +a\cdot\nabla+z)u \|_{L^p(\mathbb R^{1+n})}
\end{equation}
with $C$ independent of $a\in \mathbb C^{1+n}$ and $z\in\mathbb C$, when $(p,q)=(\frac{2(n+1)}{n+3}, \frac{2(n+1)}{n-1})$. As a consequence, they obtained a unique continuation result for $|(\partial_t^2-\Delta) u|\le |Vu|$ via the Carleman inequality
\begin{equation}\label{e:wcar}
	\|e^{v\cdot (t,x)}u\|_{L^q(\mathbb R^{1+n})} \le C\| e^{v\cdot (t,x)} (\partial_t^2-\Delta) u \|_{L^p(\mathbb R^{1+n})},
\end{equation}
which follows from \eqref{e:uni-sob}. Later, the range of $p,q$ on which the uniform Sobolev inequality \eqref{e:uni-sob} and  Carleman inequality \eqref{e:wcar} hold was completely characterized in \cite{MR3545933} and \cite{JKL18}, respectively, where the authors proved that both \eqref{e:uni-sob}  and \eqref{e:wcar} hold if and only if\footnote{If we consider the Laplacian $\Delta_{\mathbb R^{1+n}}$ instead of $\partial_t^2-\Delta$, the ranges of $p,q$ for the uniform Sobolev and Carleman inequalities do not coincide. We refer the interested readers to \cite{JKL18} for details.} 
\[	\frac1p-\frac1q=\frac2{n+1}, \quad \frac{2n(n+1)}{n^2+4n-1} <p< \frac{2n}{n+1}. 	\]
In proving the uniform Sobolev inequality \eqref{e:uni-sob}, uniform resolvent estimate (\eqref{e:uni-sob} with $a=0$) that is seemingly weaker than \eqref{e:uni-sob} is a main ingredient. So, the two are more or less equivalent. 

However, three of the authors \cite{KLS} recently proved that if the wave operator $\partial_t^2-\Delta$ is replaced by the ($(n+1)$-dimensional) Lam\'e operator $\Delta^\ast_{\mathbb R^{n+1}}$ in the above (\eqref{e:uni-sob} and \eqref{e:wcar}), then the uniform Sobolev inequality \eqref{e:uni-sob} and Carleman inequality \eqref{e:wcar} fail, while the uniform (and even non-uniform sharp) resolvent estimates are available in the general context of \cite{MR4076079} (also, see \cite{MR3060701,Co}). This shows a fundamental difference between $\Delta^\ast$ and $\Delta$.

We aim to investigate in this direction for the elastic wave operator $\partial_t^2-\Delta^\ast$.
\begin{thm} \label{t:uniresol}
There exists $C$ independent of $(a,z)\in \mathbb C\times \mathbb C$ such that 
\begin{equation}\label{e:res}
	\|u\|_{L^q(\mathbb R^{1+n})} \le C \|(\partial_t^2-\Delta^\ast + a\partial_t-z)u\|_{L^p(\mathbb R^{1+n})}
\end{equation}
if and only if 
\begin{equation}\label{e:unif-ran}
	\frac1p-\frac1q=\frac2{n+1}, \quad \frac{2n(n+1)}{n^2+4n-1} <p< \frac{2n}{n+1}. 
\end{equation}
\end{thm}
We close the introduction with a few remarks.  By the famous argument in \cite{MR894584} certain unique continuation property follows from the uniform Sobolev inequality \eqref{e:res}. See Corollary \ref{c:uniq} in Section \ref{sec:uniresol}. In view of \cite{KLS}, if the first order spatial derivatives $\nabla_x$ is involved in the right side of \eqref{e:res}, it seems that any uniform estimate cannot hold true. However, we do not pursue this issue here.

\subsubsection*{Organization}
In the next section, we diagonalize the Lam\'e operator. In Section 3, we prove the dispersive estimate (Theorem \ref{t:dispersive}). In Section 4, we prove the  Strichartz inequalities (Theorem \ref{t:strichartz}). In Section 5, we prove Theorem \ref{t:str-per}. In the last section, we prove Theorem \ref{t:uniresol} and deduce a unique continuation property of the elastic wave operator $\partial_t^2-\Delta^\ast+V$. 

\section{Diagonalization of the Lam\'e operator} \label{sec:disper}
In this section, we diagonalize $\sqrt{-\Delta^\ast}$. To do so, we need to rotate the associated multiplier and decompose the frequency space into two parts to make the involved rotations well-defined.
\subsection{Rotations in the frequency space}
Let $e_1=(1,0,\ldots, 0)^t\in\mathbb R^n$ and $S_\pm =\{\omega \in S^{n-1}\colon -\frac1{\sqrt 2} \le \omega\cdot (\pm e_1)\le 1\}$. For every $\omega\in S_\pm\setminus \{\pm e_1\}$ we set 
\[	\mathcal C_\pm(\omega)=S_\pm\cap \mathrm{span}\{e_1,\omega\}.	\]
In other words, $\mathcal C_\pm(\omega)$ is the intersection of $S_\pm$ and the great circle passing through $e_1$ and $\omega$.  We take $\rho_\pm(\omega)\in \mathrm{SO}(n)$ so that its transpose $(\rho_\pm(\omega))^t$ is the unique rotation mapping $\omega$ to $\pm e_1$ along the arc $\mathcal C_\pm(\omega)$ and satisfying $(\rho_\pm(\omega))^t y=y$ whenever $y\in \mathrm{span}\{e_1,\omega\}^\perp$. When $\omega=\pm e_1$ we set $\rho_\pm(\omega)=I_n$.

It is clear that the mapping $\rho_\pm\colon S_\pm \to \mathrm{SO}(n)$ is smooth and $\omega=\pm \rho_\pm (\omega) e_1$ if $\omega\in S_\pm$. Now, let us set $\mathbb R_{\pm}^n=\{\xi\in \mathbb R^n\setminus \{0\}\colon \xi/|\xi|\in S_\pm \} $ and define $R_{\pm}\colon \mathbb R_{\pm}^n \to \mathrm{SO}(n)$ by 
\[	R_{\pm}(\xi)=\rho_\pm (\xi/|\xi|), \quad \xi\in \mathbb R^n_\pm. 	\]
In each of $\mathbb R^n_+$ and $\mathbb R^n_-$, the matrices $R_+$ and $R_-$ satisfy the Mikhlin type condition:
\begin{lem} \label{l:mikhlin}
The mapping $R_\pm\colon \mathbb R_\pm^n \to \mathrm{SO} (n)$  is smooth and homogeneous of degree zero. Thus, every $(j,k)$-component $r^\pm_{jk}(\xi)$ of $R_\pm(\xi)$ satisfies
\begin{equation}\label{e:mikhlin}
	|\partial_\xi^\alpha r^\pm_{jk}(\xi) | \lesssim |\xi|^{-|\alpha|}, \quad \xi\in \mathbb R_\pm^n
\end{equation}
for all multi-indices $\alpha\in\mathbb N_0^n$.
\end{lem}
\begin{proof}
By the definition of $R_\pm$ and our choice of $\rho_\pm$, the smoothness,  homogeneity and boundedness of $r_{j,k}^\pm$ are clear. Hence, we need only to show the inequality \eqref{e:mikhlin}  when $|\alpha|>0$. By the standard spherical coordinate $(r,\theta)\to\xi$  and the chain rule, we see that
\[	\frac{\partial}{\partial \xi_j} = a_j(\theta) \frac{\partial}{\partial r} + A_j(\theta, \partial_\theta)\frac1r ,	\]
where $a_j$ and $A_j$ are smooth functions independent of $r$. The inequality \eqref{e:mikhlin} follows by homogeneity. 
\end{proof}

Consequently, Mikhlin's multiplier theorem imply the following. Let $\{\varphi_+, \varphi_-\}$ be a smooth partition of unity on the unit sphere $S^{n-1}$ subordinate to the  covering $\{\mathrm{int} S_+, \mathrm{int} S_-\}$ and let $\mathcal P_+$ and $\mathcal P_-$ be projections to $\mathbb R^n_+$ and $\mathbb R_-^n$, respectively, defined by 
\[	\widehat{\mathcal P_\pm f} (\xi) = \varphi_\pm (\xi/|\xi|) \widehat f(\xi), \quad \xi\in \mathbb R^n\setminus \{0\}.	\]
We denote $D=-i\nabla$ and $m(D)f(x)=(m \widehat f\,)^\vee(x)$.
\begin{lem} \label{l:norm_equi} 
Let $1<r<\infty$. Then, for any vector-valued function $f$, we have
\begin{equation} \label{e:norm_equi}
	 \|R_\pm(D)\mathcal P_\pm f\|_{L^r(\mathbb R^n)} \approx \|\mathcal P_\pm f\|_{L^r(\mathbb R^n)} \lesssim \|f\|_{L^r(\mathbb R^n)}.
\end{equation}
\end{lem}

\begin{proof}
The last inequality in the above is a direct consequence of Mikhlin's multiplier theorem. It follows from the definition \eqref{e:def1}, Lemma \ref{l:mikhlin} and Mikhlin's multiplier theorem that
\[	\|R_\pm(D)\mathcal P_\pm f\|_{L^r(\mathbb R^n)} \lesssim \|\mathcal P_\pm f\|_{L^r(\mathbb R^n)}.	\]
Considering $(R_\pm(\xi))^t$ instead of $R_\pm(\xi)$ we also get the reverse inequality
\[	\|\mathcal P_\pm f\|_{L^r(\mathbb R^n)} \lesssim \|R_\pm(D) \mathcal P_\pm f\|_{L^r(\mathbb R^n)},	\]
which shows the first estimate in \eqref{e:norm_equi}.
\end{proof}

\subsection{Diagonalization}
By taking the Fourier transform it is easy to see that the multiplier associated with the differential operator $-\Delta^\ast-z$, $z\in \mathbb C$, is the matrix-valued function $L_z\colon \mathbb R^n\to \mathcal M_{n\times n} (\mathbb C)$ defined by 
\[	L_z(\xi) := (\mu|\xi|^2-z)I_n +(\lambda+\mu)\xi\xi^t.	\]
See \cite[Proof of Lemma 2.1]{KLS}. In particular, the matrix $L(\xi):=L_0(\xi)$ is the multiplier of $-\Delta^\ast$.   

If we write $\xi\in \mathbb R^n_\pm$ as $\xi=|\xi|\omega$ with $\omega\in S_\pm$ we have $(R_\pm (\xi))^t \xi =|\xi| (\rho_\pm(\omega))^t \omega =\pm |\xi|e_1$ by the definition of $\rho_\pm$. Hence we see that, if $\xi\in \mathbb R_\pm^n$ then
\begin{align*}
	(R_\pm(\xi))^t L_z(\xi) R_\pm(\xi) 
	&= (\mu |\xi|^2-z)(R_\pm(\xi))^t R_\pm(\xi) + (\lambda+\mu) (R_\pm(\xi))^t\xi\xi^tR_\pm(\xi)\\
	&= (\mu |\xi|^2-z)I_n +(\lambda+\mu) |\xi|^2 e_1e_1^t \\
	&= \mathrm{diag} \big((\lambda+2\mu)|\xi|^2-z, \mu|\xi|^2-z, \ldots, \mu|\xi|^2-z\big) ,
\end{align*}
where $\mathrm{diag} (a_1, \ldots, a_n)$ denotes the $n\times n$ diagonal matrix whose $j$-th diagonal entry is  $a_j$. Thus, for any $\xi\in \mathbb R^n$ 
\[	\det(L(\xi)-zI_n) = ((\lambda+2\mu)|\xi|^2-z)(\mu|\xi|^2-z)^{n-1},	\]
so the eigenvalues of $L(\xi)$ are $(\lambda+2\mu)|\xi|^2$ and $\mu|\xi|^2$. The latter is of multiplicity $n-1$. 

Setting $\Lambda(\xi) :=(R_\pm(\xi))^t L(\xi) R_\pm(\xi) =\mathrm{diag} ((\lambda+2\mu)|\xi|^2, \mu|\xi|^2, \ldots, \mu|\xi|^2 )$ we have 
\begin{equation*}
	L(\xi)= R_\pm(\xi) (R_\pm(\xi))^t  L(\xi) R_\pm(\xi) (R_\pm(\xi))^t  =R_\pm(\xi) \Lambda(\xi) (R_\pm(\xi))^t
\end{equation*}
on $\mathbb R_\pm^n$, and therefore
\begin{equation}\label{e:diag}
	-\Delta^\ast=L(D)=\sum_\pm L(D)\mathcal P_\pm = \sum_\pm R_\pm (D)\Lambda(D)(R_\pm (D))^t \mathcal P_\pm.
\end{equation}
Clearly, we can take  the square root in the above equations under the assumption \eqref{e:elliptic}. So, $\sqrt{L}(\xi)= R_\pm(\xi)  \sqrt{\Lambda}(\xi) (R_\pm(\xi))^t$ on $\mathbb R_\pm^n$ and
\begin{equation}\label{e:diag2}
	 \sqrt{-\Delta^\ast}=\sqrt {L}(D)=\sum_\pm \sqrt {L}(D)\mathcal P_\pm =\sum_\pm R_\pm (D)\sqrt{\Lambda}(D)(R_\pm (D))^t\mathcal P_\pm,	
\end{equation}
where $\sqrt{\Lambda}(\xi)=\mathrm{diag} (\sqrt{\lambda+2\mu}|\xi|, \sqrt{\mu}|\xi|, \ldots, \sqrt{\mu}|\xi| )$.

\section{Proofs of the dispersive estimate and Strichartz estimate}\label{sec:pfofmain}
Considering the matrix exponential, the above diagonalization process enables us to express the semigroup as
\begin{equation}\label{dia}
	e^{it\sqrt{-\Delta^\ast}} = \sum_\pm e^{it\sqrt {L}(D)}\mathcal P_\pm  = \sum_\pm R_\pm(D) e^{it \sqrt {\Lambda}(D)} (R_\pm(D))^t \mathcal P_\pm,	
	\end{equation}
where  
\[	e^{it \sqrt {\Lambda} (D)} =\mathrm{diag} \big(e^{it\sqrt{-(\lambda+2\mu)\Delta}}, e^{it\sqrt{-\mu\Delta}}, \ldots, e^{it\sqrt{-\mu\Delta}} \big).	\]

Let $\mathcal P_0$ be a projection defined by $\widehat{\mathcal P_0f}(\xi)=\beta(|\xi|)\widehat f(\xi)$ where $\beta$ is a smooth function supported in the interval $(1/4,4)$ such that $0\le\beta\le 1$ and  $\beta(t)=1$ if $t\in[1/2,2]$. As the functions 
\[	\xi\mapsto \beta(|\xi|)\rho_\pm(\xi/|\xi|)\varphi_\pm(\xi/|\xi|), \quad  \xi\mapsto (\rho_\pm(\xi/|\xi|))^t \varphi_\pm(\xi/|\xi|)\beta(|\xi|)	\]
 are smooth and compactly supported we have the uniform bounds
\begin{equation}\label{e:smcp}
	\|\mathcal P_0 \mathcal R_\pm(D)\mathcal P_\pm h\|_{L^q}\lesssim \|h\|_{L^p}, \quad \| (\mathcal R_\pm(D))^t\mathcal P_\pm \mathcal P_0 h\|_{L^q}\lesssim \|h\|_{L^p}
\end{equation}
for any $1\le p\le q\le \infty$. Therefore, if $\widehat f$ is supported in the annulus $\{\xi\colon 1/2\le |\xi|\le 2 \}$, it follows from \eqref{dia} and \eqref{e:smcp} that 
\begin{align*}
	\|e^{it\sqrt{-\Delta^\ast}} \mathcal P_0f\|_{L^\infty(\mathbb R^n)}
	&\le \sum_\pm \| \mathcal P_0 R_\pm (D) e^{it\sqrt{\Lambda}(D)} (R_\pm(D))^t \mathcal P_\pm \mathcal P_0 f\|_{L^\infty(\mathbb R^n)}\\
	&\lesssim \sum_\pm  \|e^{it\sqrt {\Lambda}(D)} (R_\pm(D))^t \mathcal P_\pm \mathcal P_0 f\|_{L^\infty(\mathbb R^n)}.
\end{align*}
By the well-known dispersive estimate for $e^{it\sqrt{-\Delta}}$ and \eqref{e:smcp}, this is estimated by
\[	C\sum_\pm |t|^{-\frac{n-1}2} \|(R_\pm(D))^t\mathcal P_\pm \mathcal P_0 f\|_{L^1(\mathbb R^n)} \lesssim |t|^{-\frac{n-1}2} \|f\|_{L^1(\mathbb R^n)},	\]
and the proof of Theorem \ref{t:dispersive} is complete.

On the other hand, by Lemma \ref{l:norm_equi} and the Plancherel theorem it is easy to see that 
\begin{equation}\label{e:unitary}
	\|e^{it\sqrt{-\Delta^\ast}} f\|_{L^2(\mathbb R^n)} 
	\le \sum_\pm \|R_\pm (D) e^{it\sqrt{\Lambda}(D)} (R_\pm(D))^t \mathcal P_\pm f\|_{L^2(\mathbb R^n)}\lesssim \|f\|_{L^2(\mathbb R^n)}.
\end{equation}
Therefore, by the well-known theorem of Keel--Tao in \cite{MR1646048}, combining the  $L^2$ estimate and the dispersive estimate \eqref{e:dispers} implies Theorem \ref{t:strichartz}.

\subsection{Further results on the inhomogeneous Strichartz estimates}\label{sec:inhomo}
As was mentioned earlier in Remark \ref{r:inhomo}, if $f=g=0$ then the estimates \eqref{e:strichartz} are available for a wider range of the Lebesgue exponent pairs $(q,r)$ and $(\tilde q , \tilde r)$.  

For the classical wave equation \eqref{e:wave}, this phenomenon has been observed by Harmse \cite{MR1052018} and Oberlin \cite{MR1046747} for the diagonal case $q=r$ and $\tilde q=\tilde r$. Later, Foschi \cite{MR2134950} followed the scheme of Keel--Tao \cite{MR1646048} and obtained the inhomogeneous estimates for the currently known widest range  with $q\neq r$ and $\tilde q\neq\tilde r$.\footnote{For related results on the Schr\"odinger equation we refer to \cite{MR2276614, MR2684466, MR3504022}.} Furthermore, Taggart \cite{MR2719758} obtained more estimates involving the Besov spaces. More recently, Bez, Cunanan and the third author \cite{MR4052201} obtained certain weak type (in temporal variable) estimates in borderline cases.  All of these results are essentially based on the dispersive estimate  
\[	\|e^{it\sqrt{-\Delta}}\mathcal P_0f\|_{L^\infty(\mathbb R^n)} \lesssim |t|^{-\frac{n-1}2}\|f\|_{L^1(\mathbb R^n)}.	\]

In this direction, analogous results are also available for the elastic wave equation \eqref{e:elastic} as we now have the dispersive estimate \eqref{e:dispers}. To facilitate the description let us say that $(q,r)$ is \emph{wave-acceptable} if
\[	1 \leq q < \infty,\quad 2\leq r \leq \infty,\quad \frac{1}{q}<(n-1)\Big(\frac{1}{2}-\frac{1}{r}\Big), \quad \text{or}\quad (q,r)=(\infty,2).	\]
Applying the result of Foschi \cite[Theorem 1.4]{MR2134950} combined with the $L^2$ estimate \eqref{e:unitary} and the dispersive estimate \eqref{e:dispers} we have the following.
\begin{thm} \label{exinho}
Let $(q,r)$ and $(\tilde q , \tilde r)$ be wave-acceptable and suppose that $r,\tilde r <\infty$\footnote{In distinction to the statement of \cite[Theorem 1.4]{MR2134950}, the condition $r,\tilde r<\infty$ is necessary in all dimensions since we have the frequency-localized dispersive estimate \eqref{e:dispers} and need to use the Littlewood--Paley inequalities to obtain global estimates.} and
\begin{equation} \label{gap}
	\frac{1}{q}+\frac{1}{\tilde q} = \frac{n-1}{2} \Big(1-\frac{1}{r}-\frac{1}{\tilde r}\Big).
	\end{equation}
If $n>3$ we further assume that 
\[	\begin{cases}
	\frac{n-3}{r} \leq \frac{n-1}{\tilde r} , \ \ \frac{n-3}{\tilde r} \leq \frac{n-1}{r}  &\text{when} \quad \frac{1}{q}+\frac{1}{\tilde q} <1,  \\ 
	\frac{n-3}{r} < \frac{n-1}{\tilde r} , \ \ \frac{n-3}{\tilde r} < \frac{n-1}{r}, \ \ \frac{1}{r}\leq \frac{1}{q}, \ \ \frac{1}{\tilde r} \leq \frac{1}{\tilde q} &\text{when} \quad \frac{1}{q}+\frac{1}{\tilde q} =1.
\end{cases}	\]
Then, we have 
\begin{equation}\label{e:inhostr}
	\bigg\|\int_0^t \sin \big((t-s) \sqrt{-\Delta^\ast} \big) \sqrt{-\Delta^\ast}^{-1} F(s,\cdot) ds\bigg\|_{L_t^qL_x^r(\mathbb R\times \mathbb R^n)} \lesssim \|F\|_{L_t^{\tilde q'}L_x^{\tilde r'}(\mathbb R\times \mathbb R^n)}.
\end{equation}
\end{thm}

\begin{rem}
	If $(1/q^\ast,1/r^\ast)$ is the midpoint between the points $(1/q,1/r)$ and $(1/\tilde q , 1/\tilde r )$ in the theorem, then it is a sharp wave-admissible pair.
		This fact follows from the gap condition \eqref{gap} since $1/q+1/\tilde q \leq 1$.
\end{rem}

\section{An alternative proofs of Theorems \ref{t:strichartz} and \ref{exinho}}
It would be interesting to notice that the diagonalization argument provides another proofs of Theorems \ref{t:strichartz} and \ref{exinho} without passing through the dispersive estimate \eqref{e:dispers}.  

Let us first consider the homogeneous part of \eqref{e:elastic} 
\begin{equation}\label{e:elastic_homo}
\begin{cases}
	(\partial_t^2 -\Delta^*) u(t,x) = 0 , \\
	u(0,x)=f(x), \quad \partial_t u(0,x)=g(x),
\end{cases}
\end{equation}
and prove 
\begin{equation}\label{e:homo}
	\|u\|_{L_t^qL_x^r(\mathbb R\times \mathbb R^n)} \lesssim \|f\|_{\dot H^s(\mathbb R^n)}+\|g\|_{\dot H^{s-1}(\mathbb R^n)} 
\end{equation}
for $q$, $r$ and $s$ given as in Theorem \ref{t:strichartz}.  

We break \eqref{e:elastic_homo} into 
\begin{equation*}
\begin{cases}
	(\partial_t^2 -\Delta^*) u_\pm (t,x) = 0, \\
	u_\pm(0,x)=\mathcal P_\pm f (x), \quad \partial_t u_\pm (0,x)=\mathcal P_\pm g(x).
\end{cases}
\end{equation*}
By the diagonalization \eqref{e:diag2} the solutions $u_\pm$ are written as\footnote{From now on, for notational convenience, we sometimes suppress the frequency variable $\xi$ when doing so does not cause confusion.} 
\begin{equation}\label{e:homorep}
\begin{aligned}
	\widehat{u_\pm} 
	&= \cos(tR_\pm  \sqrt{\Lambda} R^t_\pm) \widehat{\mathcal P_\pm f} +\sin(tR_\pm  \sqrt{\Lambda} R^t_\pm)(R_\pm  \sqrt{\Lambda} R^t_\pm)^{-1}  \widehat{\mathcal P_\pm g}\\
	&= R_\pm \big( \cos(t\sqrt{\Lambda}) R_\pm^t\widehat{\mathcal P_\pm f} + \sin(t \sqrt{\Lambda}) \sqrt{\Lambda}^{-1} R^t_\pm  \widehat{\mathcal P_\pm g} \big),
\end{aligned}
\end{equation}
from which it is easy to see that $\mathrm{supp}\, \widehat{u_\pm} \subset \overline{\mathbb R_\pm^n}$. It is also clear that $u=u_++u_-$. 

Now it is straightforward to prove the estimate \eqref{e:homo}. Indeed, by Lemma \ref{l:norm_equi} and the classical homogeneous Strichartz estimate \cite{MR512086, MR1351643, MR1646048} we deduce 
\begin{align*}
	&\|u\|_{L_t^qL_x^r(\mathbb R\times \mathbb R^n)}
	\le  \|u_+ \|_{L_t^qL_x^r(\mathbb R\times \mathbb R^n)} + \|u_- \|_{L_t^qL_x^r(\mathbb R\times \mathbb R^n)}\\
	&\lesssim \sum_\pm \big\|\cos(t\sqrt{\Lambda}(D) ) (R_\pm (D))^t \mathcal P_\pm f +  \sin(t \sqrt{\Lambda}(D)) \sqrt{\Lambda}^{-1}(D) (R_\pm(D))^t  \mathcal P_\pm g \big\|_{L_t^qL_x^r(\mathbb R\times \mathbb R^n)} \\
	&\lesssim \sum_\pm \|(R_\pm (D))^t \mathcal P_\pm f \|_{\dot H^s(\mathbb R^n)} + \|(R_\pm (D))^t \mathcal P_\pm g\|_{\dot H^{s-1}(\mathbb R^n)} \\
	&\lesssim \| f \|_{\dot H^s(\mathbb R^n)} +\|g\|_{\dot H^{s-1}(\mathbb R^n)}.
\end{align*}
The last inequality follows from Lemma \ref{l:norm_equi} since $|\xi|^{s}I_n$ (i.e., the multiplier of $|\nabla|^{s}$ acting on $n$-dimensional vector-valued functions) commutes with all matrices. 

It remains to consider the inhomogeneous part of  \eqref{e:elastic}
\begin{equation}\label{e:elastic_inho0}
\begin{cases}
	(\partial_t^2 -\Delta^*) u(t,x) = F(t,x), \\
	u(0,x)=0, \quad \partial_t u(0,x)=0,
\end{cases}
\end{equation}
and prove the inhomogeneous estimates \eqref{e:inhostr}.  The strategy is similar to the homogeneous part. As before we break \eqref{e:elastic_inho0} into 
\begin{equation*}
\begin{cases}
	(\partial_t^2 -\Delta^*) u_\pm (t,x) = \mathcal P_\pm F (t,x), \\
	u_\pm(0,x)=0, \quad \partial_t u_\pm (0,x)=0.
\end{cases}
\end{equation*}
By the diagonalization \eqref{e:diag2} and Duhamel's formula we have 
\begin{equation}\label{e:inhoho}
\begin{aligned}
	\widehat {u_\pm}(t)
	&= \int_0^t \sin\big((t-s)R_\pm  \sqrt{\Lambda} R^t_\pm \big)(R_\pm  \sqrt{\Lambda} R^t_\pm)^{-1} \widehat {\mathcal P_\pm F}(s)ds \\
	&= R_\pm \int_0^t \sin\big((t-s) \sqrt{\Lambda}  \big) \sqrt{\Lambda}^{-1} R^t_\pm \widehat {\mathcal P_\pm F}(s)ds,	
\end{aligned}
\end{equation}
so it is clear that $\mathrm{supp}\, \widehat{u_\pm} \subset \overline{\mathbb R_\pm^n}$ and $u=\sum_\pm u_\pm$. 

By \eqref{e:inhoho}, Lemma \ref{l:norm_equi} and the inhomogeneous Strichartz estimates for the wave equation \cite{MR1052018, MR1046747, MR1646048, MR2134950} we have, for $(q,r)$ and  $(\tilde q, \tilde r)$ given in Theorem \ref{exinho}, that
\begin{align*}
	&\|u\|_{L_t^qL_x^r(\mathbb R\times \mathbb R^n)} \\
	& \lesssim \sum_\pm \bigg\| \int_0^t \sin \big( (t-s) \sqrt{\Lambda}(D) \big) \sqrt{\Lambda}^{-1}(D) (R_\pm (D))^t \mathcal P_\pm  F(s)ds  \bigg\|_{L_t^qL_x^r(\mathbb R\times \mathbb R^n)} \\
	& \lesssim \sum_{\pm} \| (R_\pm (D))^t \mathcal P_\pm   F\|_{L_t^{\tilde q'}L_x^{\tilde r'}(\mathbb R\times \mathbb R^n)} \\
	& \lesssim  \|F\|_{L_t^{\tilde q'}L_x^{\tilde r'}(\mathbb R\times \mathbb R^n)}.
\end{align*}
This completes the proof of the inhomogeneous Strichartz estimate \eqref{e:inhostr}.

\section{Perturbed equations} \label{sec:4}
In this section, we provide a new proof of Theorem \ref{t:str-per} making use of the diagonalization argument rather than using the Helmholtz decomposition.  

By Duhamel's principle, we first write the solution to \eqref{e:perturb} as 
\begin{equation} \label{e:sol}
\begin{aligned}
u (t,x)
	&= \cos( t\sqrt{-\Delta^*})f + \sin (t\sqrt{-\Delta^*})\sqrt{-\Delta^*}^{-1}g \\
	& \quad +  \int_0^t \sin((t-s)\sqrt{-\Delta^*})\sqrt{-\Delta^*}^{-1} \big[Vu(s,\cdot)\big] ds.
\end{aligned}
\end{equation}
The estimate \eqref{e:homo} gives the following \emph{a priori} estimate
\[	\big \| \cos(t\sqrt{-\Delta^*})f + \sin(t\sqrt{-\Delta^*})\sqrt{-\Delta^*}^{-1}g \big \|_{L_t^q L_x^r(\mathbb R\times \mathbb R^n)} 
	\lesssim \|f\|_{\dot H^s(\mathbb R^n)} +\|g\|_{\dot H^{s-1}(\mathbb R^n)}	\]
for $q$, $r$ and $s$ as in Theorem \ref{t:strichartz}. In this estimate if we replace $f$ and $g$ with $|\nabla|^{\frac12-s}f$ and $|\nabla|^{\frac12-s}g$, respectively, then (since the multiplier of $|\nabla|^{\frac12-s}$ commutes with all matrices) we obtain
\begin{equation}\label{e:homo1}
	\big \| \cos(t\sqrt{-\Delta^*})f + \sin(t\sqrt{-\Delta^*})\sqrt{-\Delta^*}^{-1}g \big \|_{L_t^q \dot H_r^\sigma} 
	\lesssim \|f\|_{\dot H^{1/2}} +\|g\|_{\dot H^{-1/2}} 
\end{equation}
with $\sigma=\frac12-s=\frac1q+\frac nr-\frac{n-1}2$. Hence, for the Duhamel part, we will show 
\begin{equation} \label{e:inhomo}
	\bigg\|\int_{0}^{t} \sin((t-s)\sqrt{-\Delta^*})\sqrt{-\Delta^*}^{-1} \big[Vu(s,\cdot)\big] ds \bigg\|_{L_t^q \dot H_r^\sigma} \lesssim \|V\|_{\mathcal F^p}^{1/2} \|u\|_{L^2_{x,t}(|V|)}
\end{equation}
and
\begin{equation} \label{e:weighted}
	\|u\|_{L^2_{x,t}(|V|)} \lesssim \|V\|_{\mathcal F^p}^{1/2}\big(\|f\|_{\dot H^{1/2}} + \|g\|_{\dot H^{-1/2}}\big),
\end{equation}
which are sufficient to prove Theorem \ref{t:str-per}. 

In order to prove the estimates \eqref{e:inhomo} and \eqref{e:weighted} we need the following weighted $L^2$ inequalities.
\begin{prop} \label{p:weighted}
Let $n \ge 3$ and $V$ be as in Theorem \ref{t:str-per}. Then we have
\begin{gather}
\label{p:cos}
	\|\cos (t\sqrt{-\Delta^*})f\|_{L_{x,t}^2 (|V|)} \lesssim \|V\|_{\mathcal F^p}^{1/2} \|f\|_{\dot H^{1/2}}, \\
\label{p:sin}
	\|\sin (t\sqrt{-\Delta^*})\sqrt{-\Delta^*}^{-1}g\|_{L_{x,t}^2 (|V|)} \lesssim \|V\|_{\mathcal F^p}^{1/2} \|g\|_{\dot H^{-1/2}}
\end{gather}
and 
\begin{equation} \label{p:inho}
	\bigg\| \int_0^t \sin((t-s)\sqrt{-\Delta^*})\sqrt{-\Delta^*}^{-1} F(s,\cdot)ds\bigg\|_{L^2_{x,t}(|V|)} \lesssim \|V\|_{\mathcal F^p} \|F\|_{L^2_{x,t}(|V|^{-1})}.
\end{equation}
\end{prop}
Let us hold off the proof of the proposition for the moment and first prove the estimates \eqref{e:inhomo} and \eqref{e:weighted}.

\begin{proof}[Proofs of \eqref{e:inhomo} and \eqref{e:weighted}]
Applying Proposition \ref{p:weighted} (with $F = V u$) to \eqref{e:sol}, we see that
\begin{equation*}
	\|u\|_{L^2_{x,t}(|V|)} \lesssim \|V\|_{\mathcal F^p}^{1/2}\big(\|f\|_{\dot H^{1/2}} + \|g\|_{\dot H^{-1/2}} \big) + \|V\|_{\mathcal F^p} \|u\|_{L^2_{x,t}(|V|)}.
\end{equation*}
Since $\|V\|_{\mathcal F^p}$ is small we obtain the estimate \eqref{e:weighted}.

The other estimate \eqref{e:inhomo} follows from 
\begin{equation} \label{bfCK}
	\bigg\|\int_{-\infty}^{\infty} \sin((t-s)\sqrt{-\Delta^*})\sqrt{-\Delta^*}^{-1} \big[Vu(s,\cdot)\big] ds \bigg\|_{L_t^q \dot H_r^\sigma} \lesssim \|V\|_{\mathcal F^p}^{1/2} \|u\|_{L^2_{x,t}(|V|)}
\end{equation}
by the Christ--Kiselev lemma (see \cite{MR1809116}). Furthermore, by the formula $\sin(\alpha-\beta)=\sin \alpha \cos \beta - \cos \alpha \sin \beta$, it is enough to show \eqref{bfCK} with $\sin(t\sqrt{-\Delta^*}) \cos(s\sqrt{-\Delta^*})$ and $\cos (t\sqrt{-\Delta^*})\sin(s\sqrt{-\Delta^*})$ in place of $\sin((t-s)\sqrt{-\Delta^*})$.

Making use of the estimates \eqref{e:homo1} and the dual form of \eqref{p:cos}, we see that
\begin{align*}
	&\bigg\| \sin(t\sqrt{-\Delta^*}){\sqrt{-\Delta^*}}^{-1} \int_{-\infty}^{\infty}\cos(s\sqrt{-\Delta^*}) \big[Vu(s,\cdot)\big]  ds \bigg\|_{L_t^q \dot H_r^{\sigma}} \\
	&\lesssim \bigg\| \int_{-\infty}^{\infty}\cos(s\sqrt{-\Delta^*}) \big[Vu(s,\cdot)\big]  ds \bigg\|_{\dot H^{-1/2}}
	\lesssim \|V\|_{\mathcal F^p}^{1/2} \|u\|_{L^2_{x,t}(|V|)}.
\end{align*}
Similarly, from \eqref{e:homo1} and the dual to estimate \eqref{p:sin} we deduce
\begin{align*}
	&\bigg\| \cos(t \sqrt{-\Delta^*}) \int_{-\infty}^{\infty} \sin(s\sqrt{-\Delta^*}){\sqrt{-\Delta^*}}^{-1} \big[Vu(s,\cdot)\big]  ds \bigg\|_{L_t^q \dot H_r^{\sigma}} \\
	&\lesssim \bigg\| \int_{-\infty}^{\infty}\sin(s\sqrt{-\Delta^*}){\sqrt{-\Delta^*}}^{-1} \big[Vu(s,\cdot)\big]  ds \bigg\|_{\dot H^{1/2}} \lesssim  \|V\|_{\mathcal F^p}^{1/2} \|u\|_{L^2_{x,t}(|V|)}.
\end{align*}
Thus the proof of \eqref{e:inhomo} is complete.
\end{proof}

In the rest of this section we prove Proposition \ref{p:weighted}. For the purpose we make use of weighted $L^2$ boundedness for Mikhlin multipliers and maximal functions involving the Muckenhoupt $A_q$ weights. We first recall the following (see e.g., \cite[Theorem 7.21]{MR3052498}). 

\begin{lem}\label{l:we-mul} 
Let $1<q<\infty$ and $w\in A_q$. If $m$ is a smooth function defined on $\mathbb R^n\setminus\{0\}$ satisfying 
\begin{equation*}
	|\partial_\xi^\alpha m(\xi)| \lesssim |\xi|^{-|\alpha|}
\end{equation*}
for all multi-indices $\alpha\in \mathbb N_0^n$, then we have 
\[	\|m(D)f\|_{L^q(w)} \lesssim \|f\|_{L^q(w)}.	\]
\end{lem}	

Let us then recall some useful facts on the Hardy--Littlewood maximal operator $\mathcal M$ on the class $\mathcal F^p$ (see \cite[Lemma 1]{MR1027825}): If  $V \in \mathcal{F}^{p}$ then for any $1< a < p$
\[	w=(\mathcal M(|V|^a))^{1/a} \in A_1 \cap \mathcal{F}^{p}	\]
and
\begin{equation} \label{e:WV}
\|w\|_{\mathcal{F}^p} \lesssim \|V\|_{\mathcal{F}^p}.
\end{equation}
It is also clear that $|V|\le w$ almost everywhere.  In particular, $w\in A_2$. 

\begin{proof}[Proof of Proposition \ref{p:weighted}]
Let us first prove \eqref{p:cos}. Setting $w=(\mathcal M(|V|^a))^{1/a}\in A_2$ and using Lemma \ref{l:we-mul}, we have 
\begin{equation} \label{l:weimik2}
	\|R_\pm(D)\mathcal P_\pm f\|_{L^2(w)} \approx \|\mathcal P_\pm f\|_{L^2(w)} \lesssim \|f\|_{L^2(w)}.
\end{equation}	 
Since $|V|\le w$, by using \eqref{e:homorep} with $g=0$ and \eqref{l:weimik2}, we get
\begin{align*}
	\|\cos(t\sqrt{-\Delta^\ast}) f\|_{L_{x,t}^2(|V|)}
	&\le \sum_\pm \|R_\pm(D) \cos (t\sqrt{\Lambda}(D)) (R_\pm(D))^t \mathcal P_\pm f\|_{L_{x,t}^2(w)} \\
	&\lesssim \sum_\pm \|\cos (t\sqrt{\Lambda}(D)) (R_\pm(D))^t \mathcal P_\pm f\|_{L_{x,t}^2(w)} .
\end{align*}
Now we use the analog \cite[(2.11)]{MR1309336} of \eqref{p:cos} (with $\Delta^\ast$ replaced by $\Delta$),  \eqref{e:norm_equi} and \eqref{e:WV} to estimate this by
\[	C\|w\|_{\mathcal F^p}^{1/2} \sum_\pm \|(R_\pm(D))^t \mathcal P_\pm f\|_{\dot H^{1/2}} 
	\lesssim \|V\|_{\mathcal{F}^p}^{1/2} \|f\|_{\dot H^{1/2}}.	\]
Hence, we obtain \eqref{p:cos}. The proof of  \eqref{p:sin} is similar, so  we shall omit it. 

To show \eqref{p:inho} we make use of  \eqref{e:inhoho} and \eqref{l:weimik2} to see that
\begin{align*}
	&\bigg\| \int_0^t \sin ((t-s)\sqrt{-\Delta^\ast}) \sqrt{-\Delta^\ast}^{-1} F(s,\cdot) ds \bigg\|_{L_{x,t}^2(w)} \\
	&\lesssim \sum_\pm \bigg\| \int_0^t \sin ((t-s)\sqrt{\Lambda}(D) ) \sqrt{\Lambda}^{-1}(D) (R_\pm(D))^t \mathcal P_\pm F(s,\cdot) ds \bigg\|_{L_{x,t}^2(w)} .
\end{align*}
We recall the wave equation analog \cite[Proposition 4.2]{MR1309336} of \eqref{p:inho} and utilize \eqref{e:WV} and \eqref{l:weimik2} to dominate this by 
\[	C \|w\|_{\mathcal F^p} \sum_\pm \| (R_\pm(D))^t \mathcal P_\pm F(t,\cdot)  \|_{L_{x,t}^2(w^{-1})} \lesssim \|V\|_{\mathcal F^p} \|F\|_{L_{x,t}^2(w^{-1})}.	\]
Since $|V|\le w$ the estimate \eqref{p:inho} follows. The proof of Proposition \ref{p:weighted} is complete.
\end{proof}

Finally, we present some difficult aspects and open problems related to the endpoint issue ($q=2$) in Theorem \ref{t:str-per}.
\subsubsection*{Further discussion} 
We discuss the endpoint issue $q=2$. We need the assumption $q>2$ in the proof when we apply the Christ--Kiselev lemma to handle the Duhamel term. One might be motivated to try a simple approach to use (a Lorentz space variant of) the endpoint Strichartz estimate \eqref{e:strichartz} (applied to \eqref{e:sol} with $V(x)=c|x|^{-2}$, $|c| \ll1$) 
\begin{equation}\label{e:easy-argu}
	\|u\|_{L^2_t L^{r,2}_x} \lesssim \|f\|_{\dot H^s} +\|g\|_{\dot H^{s-1}} +\| Vu\|_{L^2_tL^{\tilde r' , 2 }_x},
\end{equation}
for wave-admissible pairs $(2,r)$ and $(2,\tilde r)$ satisfying  
\begin{equation}\label{e:gapp}
	\frac12 + \frac nr=\frac n2 -s =\frac12+\frac n{\tilde r'}-2,
\end{equation}
and argue as the following\footnote{See, for example, \cite[p. 282]{MR2195116} for a similar argument concerning perturbed Schr\"odinger equations.}: If \eqref{e:easy-argu} were true, then from \eqref{e:gapp} combined with O'Neil's inequality (\cite{MR146673}) it follows that
\begin{align*}
\|u\|_{L^2_t L^{r,2}_x} \lesssim \|f\|_{\dot H^s} +\|g\|_{\dot H^{s-1}} + \|V\|_{L^{n/2,\infty}} \|u\|_{L^2_t L^{r,2}_x}
\end{align*} 
and we can ignore the last term since $\|V\|_{L^{n/2,\infty}}\ll 1$. However, unfortunately, such pairs $(2,r)$ and $(2,\tilde r)$ do not exist. 

For the wave equation perturbed by the inverse square potential, Burq et al. \cite{MR2003358} obtained the endpoint case (see Theorem 9 in \cite{MR2003358}). The framework in \cite{MR2003358} does not seem to be accessible in the elastic case because the differential operator $\nabla \mathrm{div}=(\partial^2/\partial x_i \partial x_j)_{1\le i, j\le n}$ has variable coefficients in the spherical coordinate. In fact,   
\[	\frac{\partial^2}{\partial x_i \partial x_j} =a_ia_j\frac{\partial^2}{\partial r^2} + \big(\frac{a_j}{r}A_i+\frac{a_i}{r}A_j-\frac{a_ia_j}{r}\big)\frac{\partial}{\partial r} +\frac1{r^2}A_iA_j,	\]
where $a_j$ and $A_j$ are functions of $\theta$ and $\partial_\theta$ as in the proof of Lemma \ref{l:mikhlin}.  In this regard it would be an interesting open question to ask whether the endpoint estimates hold for the elastic case. 

\section{Uniform Sobolev inequality} \label{sec:uniresol}
In this final section, we prove the uniform Sobolev inequality \eqref{e:res}, which follows from corresponding result on the wave operator in \cite{MR3545933} once we diagonalize the Lam\'e operator $\Delta^\ast$ as in the previous sections. As a corollary, the uniform inequality yields temporal unique continuation for $(\partial_t^2-\Delta^\ast)u=Vu$ whenever $V\in L^{\frac{1+n}2}(\mathbb R^{1+n})$.

\begin{proof}[Proof of Theorem \ref{t:uniresol}]
If we denote by $\mathcal F$ the space-time Fourier transform and $\mathcal F^{-1}$ the inverse of $\mathcal F$, then in terms of Fourier multiplier the inequality \eqref{e:res} is equivalent to 
\[	\left\|\mathcal F^{-1} \left\{ \big( (-\tau^2+ia\tau -z)I_n +L(\xi) \big)^{-1} \mathcal F f (\tau, \xi) \right\} \right\|_{L^q(\mathbb R^{1+n})} 
	\le C \|f\|_{L^p(\mathbb R^{1+n})}.	\]
By \eqref{e:diag} we diagonalize the multiplier and see that 
\begin{align*}
	&\big( (-\tau^2 +ia\tau -z)I_n +L(\xi) \big)^{-1} \mathcal F f (\tau, \xi) \\
	&=\sum_\pm R_\pm(\xi) \big( (-\tau^2 +ia\tau -z)I_n +\Lambda(\xi) \big)^{-1} (R_\pm(\xi))^t \varphi_\pm (\xi/|\xi|) \mathcal F f (\tau, \xi)
\end{align*}
Therefore, by making use of Lemma \ref{l:norm_equi} and the uniform Sobolev inequalities for the wave operator (\cite[Theorem 1.1]{MR3545933} with $d=1+n$) we have the uniform estimate 
\begin{align*}
	&\left\| (\partial_t^2-\Delta^\ast + a\partial_t-z)^{-1} f \right\|_{L^q(\mathbb R^{1+n})} \\
	&=\left\|\mathcal F^{-1} \left\{ \big( (-\tau^2+ia\tau -z)I_n +L(\xi) \big)^{-1} \mathcal F f (\tau, \xi) \right\} \right\|_{L^q(\mathbb R^{1+n})} \\
	&\lesssim \sum_\pm \left\| \mathcal F^{-1} \left\{ \big( (-\tau^2+ia\tau-z)I_n +\Lambda(\xi) \big)^{-1} (R_\pm(\xi))^t \varphi_\pm (\xi/|\xi|) \mathcal F f (\tau, \xi) \right\} \right\|_{L^q(\mathbb R^{1+n})} \\
	&\lesssim \sum_\pm \left\| \mathcal F^{-1} \left\{ (R_\pm(\xi))^t \varphi_\pm (\xi/|\xi|) \mathcal F f (\tau, \xi) \right\} \right\|_{L^p(\mathbb R^{1+n})} \\
	&\lesssim \|f\|_{L^p(\mathbb R^{1+n})}
\end{align*}
whenever $(p,q)$ lies in the uniform boundedness range \eqref{e:unif-ran}. This completes the proof of the uniform Sobolev inequalities \eqref{e:res}.
\end{proof}

We now turn to show the following form of Carleman inequalities with weights in the temporal variable: If $p,q$ satisfy \eqref{e:unif-ran} we have
\[	\|e^{\nu t}u\|_{L^q(\mathbb R^{1+n})} \le C\|e^{\nu t}(\partial_t^2-\Delta^\ast)u\|_{L^p(\mathbb R^{1+n})},	\]
where $C$ is independent of $\nu\in\mathbb R$.  These inequalities are direct consequences of the uniform Sobolev inequalities \eqref{e:res} since 
\[	e^{\nu t}(\partial_t^2-\Delta^\ast)e^{-\nu t}= \partial_t^2-2\nu\partial_t+\nu^2-\Delta^\ast. 	\]
As a consequence, by the well-known argument in \cite[p. 343]{MR894584} we obtain temporal unique continuation for the differential inequality
\begin{equation}\label{e:diff-ineq}
	|(\partial_t^2-\Delta^\ast) u(t,x)| \le |V(t,x)u(t,x)|.
\end{equation}
\begin{cor}\label{c:uniq}
Let $p$ satisfy the second condition in \eqref{e:unif-ran} and suppose that $V\in L^{\frac{1+n}2}(\mathbb R^{1+n})$ and $u\in W^{2,p}(\mathbb R^{1+n})$. For some $t_0\in \mathbb R$, if the support of $u$ is contained in one side of the hyperplane $\{(t,x)\in\mathbb R^{1+n}\colon t= t_0\}$ and $u$ satisfies \eqref{e:diff-ineq} almost everywhere, then $u \equiv 0$ on $\mathbb R^{1+n}$.
\end{cor}


\end{document}